\documentclass{amsart}
\usepackage{amsmath,amssymb,amsthm,amsfonts,amscd,mathrsfs}
\usepackage[all]{xy}
\usepackage[dvips]{graphicx}

\theoremstyle{plain}
\newtheorem{introtheorem}{Theorem}
\newtheorem{theorem}{Theorem}[section]
\newtheorem{proposition}[theorem]{Proposition}
\newtheorem{lemma}[theorem]{Lemma}
\newtheorem{corollary}[theorem]{Corollary}

\newtheorem*{proposition*}{Proposition}
\newtheorem{introcorollary}{Corollary}

\theoremstyle{definition}
\newtheorem{definition}[theorem]{Definition}
\newtheorem{example}[theorem]{Example}

\theoremstyle{remark}
\newtheorem{remark}[theorem]{Remark}

\newcommand{\secref}[1]{Section~\ref{#1}}
\newcommand{\thmref}[1]{Theorem~\ref{#1}}
\newcommand{\propref}[1]{Proposition~\ref{#1}}
\newcommand{\lemref}[1]{Lemma~\ref{#1}}
\newcommand{\corref}[1]{Corollary~\ref{#1}}

\def\Hom{\mathrm{Hom}}

\def\cat{\mathrm{cat}}
\def\secat{\mathrm{secat}}
\def\wgt{\mathrm{wgt}}

\def\dim{\mathrm{dim}}
\def\ker{\mathrm{ker}}
\def\nil{\mathrm{nil}}

\newcommand{\be}{\begin{enumerate}}
\newcommand{\ee}{\end{enumerate}}

\newcommand{\Z}{\mathbb{Z}}

\newcommand{\Q}{\mathbb{Q}}

\newcommand{\TC}{{\sf TC}}

\begin{document}

\title[Spaces of topological complexity one]{Spaces of topological complexity one}

\author{Mark Grant}
\author{Gregory Lupton}
\author{John Oprea}

\address{School of Mathematical Sciences, The University of Nottingham,
University Park, Nottingham, NG7 2RD, UK}

\email{mark.grant@nottingham.ac.uk}

\address{Department of Mathematics, Cleveland State University, Cleveland OH 44115 U.S.A.}

\email{g.lupton@csuohio.edu}
\email{j.oprea@csuohio.edu}

\date{\today}


\begin{abstract} We prove that a space whose topological complexity equals $1$ is homotopy equivalent to some odd-dimensional sphere.  We prove a similar result, although not in complete generality, for spaces $X$ whose higher topological complexity $\TC_n(X)$ is as low as possible, namely $n-1$.   
\end{abstract}

\thanks{This work was partially supported by a grant from the Simons Foundation (\#209575 to Gregory Lupton).}

\maketitle

\section{Introduction}
The \emph{topological complexity} of a space is a numerical homotopy invariant, of Lusternik-Schnirelmann type, introduced by Farber \cite{Far03} and motivated by the motion planning problem in the field of topological robotics.

Here, we use $\cat(X)$ to denote the Lusternik-Schnirelmann (L-S) category of $X$ (normalised, so that $\cat(S^n) = 1$), and we use $\secat(p)$ to denote the sectional category of a fibration $p$ (normalized, so  that $\secat(p) = 0$ when $p$ admits a section).  Then we define $\TC(X)$, the \emph{topological complexity}  of $X$, to be the sectional category $\secat\big( P_2\big)$ of the fibration $P_2 \colon PX \to X\times X$, which evaluates a (free) path in $X$ at its initial and final points.  See the next section for a review of terminology and precise, more verbose definitions.

The basic inequalities that relate $\cat(-)$ and $\TC(-)$ are
\begin{equation}\label{eq:basic ineq} 
\cat(X) \leq \TC(X) \leq \cat(X \times X).
\end{equation} 
It follows from the definition that we have $\cat(X) = 0$ exactly when $X$ is contractible.  It is also easy to show that $\TC(X) = 0$ exactly when $X$ is contractible.  In this paper, we consider the next step, namely when these invariants equal $1$.  As is well-known, $\cat(X) = 1$ corresponds to the case in which $X$ is a co-H-space.  This is a large class of spaces which includes all suspensions.  In addition there are well-known examples of co-H-spaces that are not suspensions.  By contrast, we find that the class of spaces with $\TC(X) = 1$ is very restricted.  By inequality (\ref{eq:basic ineq}), if $\TC(X) = 1$, then $X$ must be a co-H-space, i.e., we must have $\cat(X) = 1$.  Further, we show the following.  

\begin{introtheorem}[\corref{cor: main}]\label{main} Let $X$ be a path-connected CW complex of finite type. If \hbox{$\TC(X)=1$}, then $X$ is homotopy equivalent to some sphere $S^{2r+1}$ of odd dimension, $r \geq 0$.
\end{introtheorem}

The converse of  \thmref{main} is also true: for $r \geq 0$, we have $\TC(S^{2r+1}) = 1$ \cite[Th.8]{Far03} (note that \cite{Far03} uses ``un-normalized" $\TC(-)$, which is one more than our $\TC(-)$). 

If we also assume that $X$ is a closed manifold, then  \thmref{main} together with the positive solution to the topological Poincar\' e conjecture yield the following:

\begin{introcorollary} If $X$ is a closed manifold with $\TC(X)=1$, then $X$ is homeomorphic to some sphere of odd dimension. \qed
\end{introcorollary}

We also consider the ``higher analogues"  of topological complexity introduced by Rudyak in \cite{Rud10} (see also \cite{Rud10b} and \cite{BGRT10}).  This notion may also be motivated by a motion planning problem of a constrained type (see  \cite[Rem.3.2.7]{Rud10}).  For $n \geq 3$, we define $\TC_n(X)$, the \emph{higher topological complexity}  of $X$, as the sectional category $\secat(P_n)$ of the fibration $P_n \colon PX \to X^n$, which evaluates a (free) path in $X$ not only at its initial and final points, but also at $(n-2)$ equally timed intermediate points as well.   Again, see the next section for full definitions.  As a matter of notation, we may write $\TC(X) = \TC_2(X)$. 

The inequalities (\ref{eq:basic ineq}) extend to the following (\cite[Cor.3.3]{BGRT10}, \cite{Lu-Sc12}):
\begin{equation}\label{eq:higher basic ineq} 
\cat(X^{n-1}) \leq \TC_n(X) \leq \cat(X^n),
\end{equation} 
for $n \geq 2$.
Now, if $X$ is not contractible, then $\cat(X^{n-1}) \geq n-1$ \cite[Th.1.47]{CLOT03}.  Therefore, the next step, in the spirit of our first result, is to consider the case in which $\TC_n(X) = n-1$.  For $n\geq 3$,  there are some subtleties that arise, because of the non-straightforward way in which $\cat(-)$ may behave with respect to products.  We prove a result (\thmref{thm: TC_n = n-1}) that substantially handles this situation, and in a way that very naturally extends \thmref{main}.  For instance, our result includes the following.

\begin{introtheorem}\label{main:higher} Let $X$ be a path-connected CW-complex of finite type.  Suppose that we have  \hbox{$\TC_n(X)=n-1$}, some $n \geq 3$.  If $X$ is simply connected, then $X$ is homotopy equivalent to some odd-dimensional sphere $S^{2r+1}$,   $r \geq 1$.  If $\pi_1(X) \not= \{e\}$, and $X$ is a nilpotent space, then $X$ is homotopy equivalent to the circle $S^1$.
\end{introtheorem}

Conversely, it is known that, for $r \geq 0$, we have $\TC_n(S^{2r+1}) = n-1$ (\cite[Sec.4]{Rud10}---note that un-normalized $\TC(-)$ is used there).  This completely describes the situation in the nilpotent case.  Notice that our results imply that, for $X$ nilpotent, if $\TC_n(X)=n-1$ for some $n \geq 2$, then we have $\TC_n(X)=n-1$ \emph{for all}  $n \geq 2$.   Our actual results do give partial information about the general, non-nilpotent situation (see \thmref{thm: TC_n = n-1} for details).  Once more, if we also assume that $X$ is a closed manifold, then we may replace ``homotopy equivalent'' in the conclusions of \thmref{main:higher} by ``homeomorphic." 

From our results, we identify precisely how (higher) topological complexity behaves for co-H-spaces:

\begin{introcorollary}[\corref{cor: higher TC Co-H}]
Let $X$ be a non-contractible, path-connected CW complex of finite type.  If $X$ is a  co-H-space, then either (a) $X$ is of the homotopy type of some odd-dimensional sphere, and we have $\TC_n(X) = n-1$ for all $n \geq 2$; or (b) we have $\TC_n(X) =  n$ for all $n \geq 2$.  
\end{introcorollary}

The paper is organized as follows.  In \secref{sec: Prelim} we review basic definitions and vocabulary, and establish two intermediate results:   \propref{prop: homology sphere} is basic for what follows;  \propref{prop: pi_1 finite} gives an interesting lower bound for $\cat(X^n)$.  In \secref{sec: TC = 1}, we prove our main result about $\TC_n(X) = n-1$, from which we conclude our result about  $\TC(X) = 1$.    

\section{Definitions and Preliminary Results}\label{sec: Prelim}

We refer to \cite{CLOT03} for a general introduction to L-S category and related topics, such as sectional category.  Here, we recall that  $\cat(X)$ is the smallest $n$ for which there is an open covering $\{ U_0, \ldots, U_n \}$ by $(n+1)$ open sets, each of which is contractible in $X$.  
The \emph{sectional category} of a fibration $p \colon E \to B$, denoted by $\secat(p)$, is the smallest number $n$ for which there is an open covering $\{ U_0, \ldots, U_n \}$ of $B$ by $(n+1)$ open sets, for each of which there is a local section $s_i \colon U_i \to E$ of  $p$, so that $p\circ s_i = j_i \colon U_i \to B$, where $j_i$ denotes the inclusion.

 Let $PX$ denote the space of (free) paths on a space $X$.  There is a fibration $P_2\colon PX \to X\times X$, which evaluates a path at initial and final point: for $\alpha \in PX$, we have $P_2(\alpha) = \big(\alpha(0), \alpha(1)\big)$.  This is a fibrational substitute for the diagonal map $\Delta \colon X \to X \times X$.  We define the \emph{topological complexity} $\TC(X)$ of $X$ to be the sectional category $\secat\big( P_2\big)$ of this fibration.  That is, $\TC(X)$ is the smallest number $n$ for which there is an open cover $\{ U_0, \ldots, U_n \}$ of $X \times X$ by $(n+1)$ open sets, for each of which there is a local section $s_i \colon U_i \to PX$ of  $P_2$, i.e., for which $P_2\circ s_i = j_i \colon U_i \to X \times X$, where $j_i$ denotes the inclusion.

More generally,  let $n \geq 2$ and consider the fibration
$$P_n \colon PX \to X \times \cdots \times X = X^n,$$
defined by dividing the unit interval $I = [0, 1]$ into $(n-1)$ subintervals of equal length, with $n$ subdivision points $t_0 = 0, t_1 = 1/(n-1), \ldots, t_{n-1} = 1$ (thus $(n-2)$ subdivision points interior to the interval), and then evaluating at each of the $n$ subdivision points, thus:
$$P_n(\alpha) = \big(  \alpha(0), \alpha(t_1), \ldots, \alpha(t_{n-2}), \alpha(1)\big),$$
for $\alpha \in PX$.  This is a fibrational substitute for the $n$-fold diagonal $\Delta_n\colon X \to X^n$.   Then the \emph{higher topological complexity} $\TC_n(X)$ is defined as $\TC_n(X) = \secat(P_n)$.

Let $H_*(X)$, respectively $\widetilde{H}_*(X)$, denote homology, respectively reduced homology, with integer coefficients.
By $\dim_{\Bbbk}\big(\widetilde{H}^*(X;\Bbbk)\big)$, we mean the dimension as a graded $\Bbbk$-vector space of the reduced cohomology of $X$ with coefficients in the field $\Bbbk$. In this paper, by an \emph{integral homology sphere}, we mean a space $X$ with integral homology isomorphic to that of $S^n$ for some $n \geq 1$.  (Note that, here, we do not implicitly assume that $X$ is a manifold.)  By a CW complex of \emph{finite type}, we mean one that has finitely many cells of each dimension.  Note that a CW complex of finite type has integral homology group $H_i(X)$ a finitely generated abelian group, for each $i$.   In the proof of the following result, and in the sequel, we make use of the universal coefficient theorem for cohomology (UCT), as given in \cite[Th.3.2]{Hat}, for instance.

\begin{proposition}\label{prop: homology sphere}
Let $X$ be a path-connected CW complex of finite type.  Suppose that $\dim_{\Bbbk}\big(\widetilde{H}^*(X;\Bbbk)\big) \leq 1$ for all choices of field $\Bbbk$.  Then either $X$ is acyclic, or $X$ is an integral homology sphere. 
\end{proposition} 

\begin{proof}
If $X$ is acyclic, then $\widetilde{H}_*(X) = 0$ and hence $\widetilde{H}^*(X;\Bbbk) = 0$ for all choices of field $\Bbbk$.  
So suppose that $X$ is not acyclic, and let $H_r(X)$ be the first non-trivial homology group of $X$, $r \geq 1$.  Since $X$ is of finite type,  we may write
\[
H_r(X)\cong \Z^n\qquad\mbox{or}\qquad H_r(X)\cong \Z^n\oplus\Z/p^{k}\oplus\Z/p_1^{k_1}\oplus\cdots\oplus\Z/p_\ell^{k_\ell}
\]
for some rank $n\geq 0$, primes $p\leq p_1\leq\cdots \le p_\ell$ and natural numbers $k, k_1,\ldots, k_\ell$.

First suppose the torsion part of $H_r(X)$ is non-trivial, so that at least the summand $\Z/p^{k}$ is non-zero.  By the UCT we have
\[
H^r(X;\Z/p)\cong \operatorname{Hom}(H_r(X),\Z/p)\cong \Z/p\oplus S,
\]
\[
H^{r+1}(X;\Z/p)\supseteq\operatorname{Ext}(H_r(X),\Z/p)\cong\Z/p\oplus T,
\]
where $S$ and $T$ are some finite $\Z/p$-vector spaces. It then follows that we have $\dim_{\Z/p}\big(\widetilde{H}^*(X;\Z/p)\big) \geq 2$, which contradicts our assumption. Thus $H_r(X)$ is torsion-free.

Now suppose that $H_r(X) \cong \Z^n$ with rank $n\geq 2$. Then $H^r(X;\Q)\cong\Q^n$ is of dimension at least $2$, which again contradicts our assumption. We conclude that $H_r(X)\cong \Z$.

Now consider homology groups in higher degrees, starting with $H_{r+1}(X)$.   Because $H_r(X)\cong \Z$, a similar argument to the above, using the UCT and then rational coefficients,  shows that $H_{r+1}(X) = 0$.  Then, arguing inductively, one sees that $H_i(X)=0$ for all $i>r$, and thus $X$ is an integral homology $r$-sphere.
\end{proof}

Our next result seems of interest in its own right, as a general statement about the L-S category of products.  In its proof, we use the notion of the \emph{category weight} of a cohomology class, which is commonly used to obtain lower bounds on L-S category.  A general discussion of this notion is given in \cite[Sec.2.7, Sec.8.3]{CLOT03}.

\begin{proposition}\label{prop: pi_1 finite} 
Let  $X$ be a path-connected CW complex whose fundamental group has a non-trivial element of finite order.  Then we have $\cat(X^n) \geq 2n$, for each $n \geq 1$.
\end{proposition}

\begin{proof}  By assumption, $\pi_1(X)$ has an element of prime order.  Hence, we may choose a cover $Y$ of $X$ whose fundamental group is the cyclic group $\Z/p$, with $p$ a prime.  Note that we then have $H_1(Y) \cong \Z/p$ also.  We look at the long exact cohomology sequence associated to the short exact sequence of coefficients 
$$\xymatrix{ 0 \ar[r] & \begin{displaystyle} \Z\slash p\end{displaystyle} \ar[r]^{\times p} & \begin{displaystyle}\Z/p^2\end{displaystyle} \ar[r]^{r_p} & \begin{displaystyle}\Z/p\end{displaystyle} \ar[r] & 0}$$
in which $r_p$ denotes  reduction mod $p$ (see, e.g.~\cite[Sec.3.E]{Hat}).   From the UCT, we have  $H^1(Y; \Z/p^2) \cong \Hom(\Z/p,\Z/{p^2}) \cong \Z/p$  (no finiteness assumptions on $Y$ are required here), and likewise $H^1(Y; \Z/p) \cong \Z/p$.  

The map  $H^1(Y;\Z/{p^2}) \to H^1(Y;\Z/p)$ induced by $r_p$ is zero. Therefore the Bockstein $\beta \colon H^1(Y;\Z/p) \to H^2(Y;\Z/p)$ is injective.  Set $y = \beta(x) \in H^2(Y;\Z/p)$, where $x \in H^1(Y;\Z/p)$ is a generator.  By \cite[Cor.4.7]{Rud99}, (see also \cite{Fa-Hu92}) the class $y$ has (strict, or essential) category weight at least $2$.  (Actually, \cite[Cor.4.7]{Rud99} is stated for odd primes $p$.  But if $p=2$, then we have $\beta(x) = Sq^1(x) = x\cup x$, which is certainly of weight $2$.)

Now note that the cross product $y \times y = p_1^*(y) \cup p_2^*(y) \in H^4(Y\times Y;Z/p)$ has category weight at least $4$.  This follows from standard properties of category weight, as summarized, for example, in \cite[Prop.8.22]{CLOT03}:  Here, $p_1, p_2\colon Y \times Y \to Y$ denote the projections onto either factor, and we denote by $\wgt(u)$ the (strict, or essential) category weight of a cohomology class $u$.  Then we have 
\begin{align*}
\wgt(y\times y) &= \wgt\big( p_1^*(y) \cup p_2^*(y)  \big) \geq \wgt\big( p_1^*(y)\big) +  \wgt\big( p_2^*(y)  \big) \\
&\geq \wgt( y)+  \wgt(y) = 4.    
\end{align*}
That the cross product $y \times y$ is nonzero follows from \cite[VII.Ex.7.15(1)]{Dold95}, since $Z/p$ is a field. By an easy inductive argument, we also have that the $n$-fold cross product $y \times \dots \times y \in H^{2n}( Y^n;Z/p)$ is nonzero, and has category weight at least $2n$.
 
Therefore, by \cite[Prop.8.22]{CLOT03}, we have $\cat(Y^n) \geq 2n$.  Hence,  since $Y^n$ covers $X^n$, and therefore  $\cat(Y^n) \leq \cat(X^n)$ (see \cite[Cor.1.45]{CLOT03}), we have $\cat(X^n)\geq 2n$ also.
\end{proof}

\begin{example}
Let $P$ denote the Poincar{\'e} $3$-sphere and $P^*$ denote its $2$-skeleton.  Then $P$ is an integral homology $3$-sphere, and $P^*$ is an acyclic space.  One might wonder whether $\TC_n(P) = n-1$ or $\TC_n(P^*) = n-1$ for some $n$.  However, $\pi_1(P) \cong \pi_1(P^*)$ is not torsion-free---it is a finite group of order $120$, in fact.  Hence \propref{prop: pi_1 finite} implies that both $\cat\big((P)^{n-1}\big)$ and $\cat\big((P^*)^{n-1}\big)$ are at least $2n-2$, for each $n \geq 2$, and thus $\TC_n(P) , \TC_n(P^*) \geq 2n-2 > n-1$, for each $n \geq 2$. 
\end{example}

\section{Spaces of lowest possible (higher) topological complexity}\label{sec: TC = 1}

We begin by recalling the standard cohomological lower bound for $\TC_n(-)$, which will be used in the sequel.  

\begin{definition} Let $\Bbbk$ be a field. The homomorphism induced on cohomology with coefficients in $\Bbbk$ by the $n$-fold diagonal $\Delta_n \colon X \to X^n$ (and thus by $P_n \colon PX \to X^n$, which is a fibrational substitute for it) may be identified with the $n$-fold cup product homomorphism
\[
\cup_n(X) \colon H^*(X;\Bbbk)\otimes_\Bbbk \cdots \otimes_\Bbbk H^*(X;\Bbbk) \to  H^*(X;\Bbbk).
\]
The \emph{ideal of $n$-fold zero divisors}  is $\ker \cup_n(X)$, the kernel of $\cup_n(X)$.
The \emph{$n$-fold zero-divisors cup-length} is $\nil\big(\ker \cup_n(X)\big)$, the nilpotency of this ideal, which is to say the number of factors in the longest non-trivial product of elements from this ideal.
\end{definition}

\begin{proposition}[{\cite[Th.7]{Far03}, \cite[Prop.3.4]{Rud10}, \cite[Th.3.9]{BGRT10}}]%
\label{prop: zero divisors lower bd}%
For any field $\Bbbk$, we have $\nil\big(\ker \cup_n(X)\big) \leq \TC_n(X)$. \qed
\end{proposition}

For an element $a \in \widetilde{H}^*(X;\Bbbk)$, we write $\bar{a} = a\otimes 1 - 1 \otimes a \in H^*(X;\Bbbk) \otimes H^*(X;\Bbbk)$.  Clearly, $\bar{a}$ is a non-zero element in the ideal of $2$-fold zero divisors.
We adopt notation from the proof of \cite[Th.3.14]{BGRT10} to describe certain $n$-fold zero divisors.  For $i = 1, \ldots, n$, let $p_i \colon X^n \to X$ denote projection on the $n$th factor.  Then we write  $a_i = (p_i)^*(a) \in  H^*(X^n; \Bbbk)$, which we regard as an element of $H^*(X; \Bbbk)^{\otimes n}$ under the identification $H^*(X^n; \Bbbk) \cong H^*(X; \Bbbk)^{\otimes n}$, namely, the K{\"u}nneth theorem.   Then we have the $n-1$ elements $\{a_1 - a_2, a_1 - a_3, \ldots, a_1 - a_n \}$, each of which is an $n$-fold zero divisor in $H^*(X)^{\otimes n}$.

\begin{lemma}\label{lem: higher zero divisors}
Suppose we have $a, b \in H^*(X)$.  With the above notation, for $n \geq 2$, we have 
$$(a_1 - a_2)(a_1 - a_3) \cdots (a_1 - a_n) \equiv (-1)^n \big( a\otimes 1 - 1 \otimes a \big) \otimes a \otimes \cdots \otimes a$$
modulo terms in the ideal of $H^*(X) ^{\otimes n}$ generated by the elements $a^2 \otimes 1 \otimes \cdots \otimes 1$ and $a \otimes \cdots \otimes a \otimes 1$.   Consequently,  we have
$$(b_1 - b_2)(a_1 - a_2)(a_1 - a_3) \cdots (a_1 - a_n) \equiv (-1)^{n+1} \big(   b\otimes a  + (-1)^{|a| |b|} a \otimes b \big) \otimes a \otimes \cdots \otimes a,$$
modulo terms in the ideal of $H^*(X) ^{\otimes n}$ generated by the elements $a^2 \otimes 1 \otimes \cdots \otimes 1$, $ba \otimes 1 \otimes \cdots \otimes 1$, and $1 \otimes ba \otimes 1 \otimes \cdots \otimes  1$.
\end{lemma}

\begin{proof}
We proceed by induction, with the induction hypothesis that, for $2 \leq k \leq n$, we have 
$$(a_1 - a_2)(a_1 - a_3) \cdots (a_1 - a_k) \equiv (-1)^k \big( a\otimes 1 - 1 \otimes a \big) \otimes a \otimes \cdots \otimes a \otimes 1 \otimes \cdots \otimes 1$$
modulo terms in the ideal $I_k$ of $H^*(X) ^{\otimes n}$ generated by the elements $a^2 \otimes 1 \otimes \cdots \otimes 1$ and $a \otimes \cdots \otimes a \otimes 1 \cdots \otimes 1$, where we have $(k-1)$ occurrences of $a$ in each term of the displayed element and in the latter ideal generator.   Induction starts with $k = 2$, where there is nothing to prove.  For the induction step, we use the induction hypothesis to write $\big((a_1 - a_2)(a_1 - a_3) \cdots (a_1 - a_{k})\big)(a_1 - a_{k+1})$ as
\begin{align*}
 &\equiv  (-1)^k  \big((a\otimes 1 - 1 \otimes a ) \otimes a \otimes \cdots \otimes a \otimes 1 \otimes \cdots \otimes 1\big) \big(a\otimes1\otimes \cdots \otimes 1\big)\\
 & - (-1)^k  \big((a\otimes 1 - 1 \otimes a ) \otimes a \otimes \cdots \otimes a \otimes 1 \otimes \cdots \otimes 1\big) \big(1 \otimes \cdots \otimes 1 \otimes a \otimes 1 \otimes \cdots \otimes 1\big).
 \end{align*}
The first part of this expression contributes
$$\pm (a^2\otimes 1 \pm a \otimes a ) \otimes a \otimes \cdots \otimes a \otimes 1 \otimes \cdots \otimes 1,$$
which is  in the ideal $I_{k+1}$.  The second part contributes 
$$(-1)^{k+1} \big( a\otimes 1 - 1 \otimes a \big) \otimes a \otimes \cdots \otimes a \otimes 1 \otimes \cdots \otimes 1,$$
with $k$ occurrences of $a$ in each term.  This completes the induction step, and the first assertion follows.

For the second assertion, observe that we may write 
$$(b_1 - b_2)(a_1 - a_2)(a_1 - a_3) \cdots (a_1 - a_n),$$
from the first part,  as congruent to
$$(-1)^n \big( (b\otimes 1 - 1 \otimes b)( a\otimes 1 - 1 \otimes a) \big) \otimes a \otimes \cdots \otimes a$$
modulo terms in the ideal $\big((b\otimes1 - 1 \otimes b)\otimes 1 \otimes \cdots \otimes 1 \big) I_n$.
But now it is clear that the only contribution outside the ideal generated by $a^2 \otimes 1 \otimes \cdots \otimes 1$, $ba \otimes 1 \otimes \cdots \otimes 1$, and $1 \otimes ba \otimes 1 \otimes \cdots \otimes  1$ is that asserted.
\end{proof}

\begin{theorem}\label{thm: TC_n = n-1}
Let $X$ be a path-connected CW complex of finite type.    If $\TC_n(X) = n-1$, for some $n \geq 2$, then $\pi_1(X)$ is torsion-free and either $X$ is acyclic or $X$ is an odd-dimensional integral  homology sphere.  
Furthermore, we have:
\begin{itemize}
\item[(A)] if $X$ is simply connected, then for some $r \geq 1$ we have $X \simeq S^{2r+1}$;
\item[(B)] if $\pi_1(X) \not= \{e\}$ and if $X$ is a nilpotent space, then we have $X \simeq S^1$; and 
\item[(C)] if $\pi_1(X) \not= \{e\}$, and if $X$ is a co-H-space, then we have $X \simeq S^1$.
\end{itemize}  
\end{theorem}

\begin{proof}
By combining \propref{prop: pi_1 finite} with the first inequality of (\ref{eq:higher basic ineq}), we conclude that $\pi_1(X)$ must be torsion-free.

Next, we show that $X$ satisfies the hypotheses of \propref{prop: homology sphere}.  For this, we argue by contradiction. Suppose that,  for some field $\Bbbk$, we have $a\in H^r(X;\Bbbk)$ and $b\in H^s(X;\Bbbk)$ with $r,s>0$,  and $\{a,b\}$ linearly independent over $\Bbbk$. By \lemref{lem: higher zero divisors}, the $n$-fold product of $n$-fold zero-divisors
$$(b_1 - b_2)(a_1 - a_2)(a_1 - a_3) \cdots (a_1 - a_n)$$
is congruent to 
$$(-1)^{n+1} \big(   b\otimes a  + (-1)^{|a| |b|} a \otimes b \big) \otimes a \otimes \cdots \otimes a,$$
modulo terms in the ideal of $H^*(X) ^{\otimes n}$ generated by the elements $a^2 \otimes 1 \otimes \cdots \otimes 1$, $ba \otimes 1 \otimes \cdots \otimes 1$, and $1 \otimes ba \otimes 1 \otimes \cdots \otimes  1$.  It follows that this term
is nonzero, as  $a$ and $b$ are linearly independent, and so we have $\TC_n(X) \geq n$ by \propref{prop: zero divisors lower bd}, which is a contradiction.  

Therefore, $X$ must satisfy the hypothesis of \propref{prop: homology sphere}, and either $X$ is acyclic, or $X$ is an integral homology sphere.  If $X$ is an even dimensional integral homology sphere, however, then once  again by \lemref{lem: higher zero divisors}, the $n$-fold product of $n$-fold zero-divisors
$$(a_1 - a_2)(a_1 - a_2)(a_1 - a_3) \cdots (a_1 - a_n) $$
is congruent to 
$$(-1)^{n+1} \big(  a\otimes a  + (-1)^{|a| |a|} a \otimes a \big) \otimes a \otimes \cdots \otimes a
= (-1)^{n+1}\,2\, a \otimes  \cdots \otimes a,$$
modulo terms in the ideal of $H^*(X) ^{\otimes n}$ generated by the elements $a^2 \otimes 1 \otimes \cdots \otimes 1$ and $1 \otimes a^2 \otimes 1 \otimes \cdots \otimes  1$.  We may take rational coefficients here, for example, and then we have $\TC_n(X) \geq n$ by \propref{prop: zero divisors lower bd}, which is again a contradiction.  
The only possibilities that remain, then, are that $X$ is acyclic or $X$ is an odd-dimensional integral homology sphere.

We treat the remaining cases separately.

(A)  Assume that  $X$ is simply connected.  Then $X$ cannot be acyclic. Indeed, Whitehead's Theorem would then imply that $X$ were contractible, and hence we would have $\TC_n(X)=0$.  Therefore, $X$ is an odd-dimensional integral homology sphere.  But any simply connected integral homology sphere is of the homotopy type of the sphere (of the same dimension), by the theorems of Hurewicz and Whitehead.     

(B)  Suppose that $X$ is a nilpotent space with $\pi_1(X) \not= \{e\}$.  Since $\pi_1(X)$ is nilpotent, we cannot have $H_1(X) = 0$.  Therefore, the only possibility is that  $X$ is an integral homology circle.    So   let $j\colon S^1 \to X$ be a generator of $\pi_1(X)$ that, under  the Hurewicz homomorphism $h \colon \pi_1(X) \to H_1(X)$, is mapped to the generator $1 \in H_1(X) \cong \Z$.  Then we have $j_* \colon H_1(S^1) \to H_1(X)$ is an isomorphism, since both groups are isomorphic to $\Z$.  However, $H_i(S^1)$ and $H_i(X)$ are both zero for $i \geq 2$, and thus $j\colon S^1 \to X$ is an integral homology equivalence.  As both $S^1$ and $X$ are nilpotent spaces, it follows from  \cite{Dror} (see  also \cite{Ger75}) that $j \colon S^1 \to X$ is a homotopy equivalence.

(C) Finally, suppose  that $X$ is a co-H-space with $\pi_1(X) \not= \{e\}$.  
We claim that $\pi_1(X)$ must be isomorphic to $\Z$.
For, as a co-H-space, $X$ must have free fundamental group. Since $X$ is of finite type, $\pi_1(X)$ must be a finitely-generated free group, and hence isomorphic to a free product of $k$ copies of $\Z$, for some positive integer $k$.
If $k\geq 2$, then the rational cohomology group $H^1(X;\Q)\cong \operatorname{Hom}(H_1(X);\Q)\cong \operatorname{Hom}(\Z^n,\Q)\cong\Q^n$ has dimension at least $2$, which contradicts the first part of the present theorem.  Thus we have $\pi_1(X) \cong \Z$.
Now, since $X$ is a co-H-space, it follows from \cite[Theorem 2.1]{ISS} that the universal cover $\tilde{X}$ of $X$ has trivial reduced integral homology.    Note that the hypotheses of  \cite[Theorem 2.1]{ISS} are automatically satisfied if $X$ is a co-H space.  Also, that result is stated for $X$ finite, but it is deduced as an immediate consequence of Theorem 4.4 of the same paper, which only requires $X$ of finite type.  In the present situation, we have $H_*(X) = 0$ for $*>1$, as $X$ is a homology circle, which gives $H_*(\widetilde{X}) = 0$ for $*>1$, by the expression relating these two in \cite[Theorem 2.1]{ISS}.  Hence, $X$ has acyclic universal cover, which is thus contractible by Whitehead's Theorem. Hence $X$ is a $K(\Z,1)$-space, that is,  homotopy equivalent to the  circle.  
\end{proof}

\begin{corollary}\label{cor: main} 
Let $X$ be a path-connected CW complex of finite type. If \hbox{$\TC(X)=1$}, then $X$ is homotopy equivalent to some sphere of odd dimension.
\end{corollary}

\begin{proof}
The inequalities $\cat(X)\leq\TC(X)=1$ imply that $X$ is a co-H-space.  The result follows from parts (A) and (C) of \thmref{thm: TC_n = n-1}.
\end{proof}

\begin{corollary}\label{cor: higher TC Co-H}
Let $X$ be a path-connected, non-contractible CW complex of finite type.  If $X$ is a co-H-space, then either (a) $X$ is of the homotopy type of some odd-dimensional sphere, and we have $\TC_n(X) = n-1$ for all $n \geq 2$; or (b) we have $\TC_n(X) =  n$ for all $n \geq 2$.  \qed
\end{corollary}

\begin{remark}
The combination of \propref{prop: pi_1 finite} and the first inequality of (\ref{eq:higher basic ineq}) actually implies the following:  If $n-1 \leq \TC_n(X) \leq 2n -3$, then $\pi_1(X)$ is torsion-free. This fact suggests that merely requiring a small value of $\TC_n(X)$---as opposed to requiring that it equal the lower bound from the first inequality of (\ref{eq:higher basic ineq})---already entails strong restrictions on the topology of a space.
\end{remark}


\begin{thebibliography}{10}

\bibitem{BGRT10}
I.~Basabe, J.~Gonz{\'a}lez, Y.~Rudyak, and D.~Tamaki, \emph{Higher topological
  complexity and homotopy dimension of configuration spaces on spheres},
  preprint, arXiv:1009.1851v5 [math.AT], 2010.

\bibitem{CLOT03}
O.~Cornea, G.~Lupton, J.~Oprea, and D.~Tanr{\'e},
  \emph{Lusternik-{S}chnirelmann category}, Mathematical Surveys and
  Monographs, vol. 103, American Mathematical Society, Providence, RI, 2003.
  \MR{1990857 (2004e:55001)}

\bibitem{Dold95}
Albrecht Dold, \emph{Lectures on algebraic topology}, Classics in Mathematics,
  Springer-Verlag, Berlin, 1995, Reprint of the 1972 edition. \MR{1335915
  (96c:55001)}

\bibitem{Dror}
Emmanuel Dror, \emph{A generalization of the {W}hitehead theorem}, Symposium on
  {A}lgebraic {T}opology ({B}attelle {S}eattle {R}es. {C}enter, {S}eattle,
  {W}ash., 1971), Springer, Berlin, 1971, pp.~13--22. Lecture Notes in Math.,
  Vol. 249. \MR{0350725 (50 \#3217)}

\bibitem{Fa-Hu92}
Edward Fadell and Sufian Husseini, \emph{Category weight and {S}teenrod
  operations}, Bol. Soc. Mat. Mexicana (2) \textbf{37} (1992), no.~1-2,
  151--161, Papers in honor of Jos{\'e} Adem (Spanish). \MR{1317569
  (95m:55007)}

\bibitem{Far03}
M.~Farber, \emph{Topological complexity of motion planning}, Discrete Comput.
  Geom. \textbf{29} (2003), no.~2, 211--221.

\bibitem{Ger75}
S.~M. Gersten, \emph{The {W}hitehead theorem for nilpotent spaces}, Proc. Amer.
  Math. Soc. \textbf{47} (1975), 259--260. \MR{0365563 (51 \#1815)}

\bibitem{Hat}
Allen Hatcher, \emph{Algebraic topology}, Cambridge University Press,
  Cambridge, 2002. \MR{1867354 (2002k:55001)}

\bibitem{ISS}
Norio Iwase, Shiroshi Saito, and Toshio Sumi, \emph{Homology of the universal
  covering of a co-{H}-space}, Trans. Amer. Math. Soc. \textbf{351} (1999),
  no.~12, 4837--4846. \MR{1487618 (2000c:55013)}

\bibitem{Lu-Sc12}
G.~Lupton and J.~Scherer, \emph{Topological complexity and {H}-spaces}, To
  appear, Proc. A.M.S., arXiv:1106.3399v1[math.AT], 2011.

\bibitem{Rud99}
Y.~Rudyak, \emph{On category weight and its applications}, Topology \textbf{38}
  (1999), no.~1, 37--55. \MR{1644063 (99f:55007)}

\bibitem{Rud10}
\bysame, \emph{On higher analogs of topological complexity}, Topology Appl.
  \textbf{157} (2010), no.~5, 916--920. \MR{2593704}

\bibitem{Rud10b}
\bysame, \emph{Erratum to ``{O}n higher analogs of topological complexity''
  [{T}opology {A}ppl. 157 (5) (2010) 916--920] [mr2593704]}, Topology Appl.
  \textbf{157} (2010), no.~6, 1118. \MR{2593724}

\end{thebibliography}

\providecommand{\bysame}{\leavevmode\hbox to3em{\hrulefill}\thinspace}
\providecommand{\MR}{\relax\ifhmode\unskip\space\fi MR }
\providecommand{\MRhref}[2]{%
  \href{http://www.ams.org/mathscinet-getitem?mr=#1}{#2}
}
\providecommand{\href}[2]{#2}

\end{document}